\documentclass[11pt]{amsart}
\usepackage{amsmath, amscd, amsthm, amssymb,bm}
\usepackage{graphicx}

\usepackage{fullpage}

\def\C{{\mathbf C}}

\def\A{{\mathbf A}}

\newtheorem{theorem}{Theorem}[section]
\newtheorem{thm}[theorem]{Theorem}
\newtheorem{lemma}[theorem]{Lemma}
\newtheorem{lem}[theorem]{Lemma}
\newtheorem{proposition}[theorem]{Proposition}
\newtheorem{prop}[theorem]{Proposition}

\theoremstyle{definition}

\theoremstyle{remark}
\newtheorem{remark}[theorem]{Remark}

\renewcommand{\(}{\left(}
\renewcommand{\)}{\right)}
\newcommand{\set}[1]{\left\lbrace#1\right\rbrace}

\newcommand{\mm}[4]{\(\begin{smallmatrix} #1 & #2\\ #3 & #4\end{smallmatrix}\)}
\newcommand{\mb}[4]{\(\begin{array}{cc} #1 & #2\\ #3 & #4\end{array}\)}

\DeclareMathOperator{\gal}{Gal}
\DeclareMathOperator{\diag}{diag}
\DeclareMathOperator{\tr}{tr}
\DeclareMathOperator{\ind}{Ind}
\DeclareMathOperator{\SO}{SO}

\DeclareMathOperator{\GSp}{GSp}

\DeclareMathOperator{\PGU}{PGU}
\DeclareMathOperator{\GU}{GU}

\DeclareMathOperator{\SL}{SL}
\DeclareMathOperator{\GL}{GL}
\DeclareMathOperator{\PGL}{PGL}

\begin{document}

\title{Multivariate Rankin-Selberg integrals on $\GL_4$ and $\GU(2,2)$}
\author{Aaron Pollack}
\address{Department of Mathematics, Institute for Advanced Study, Princeton, NJ 08540, USA}\email{aaronjp@math.ias.edu}
\author{Shrenik Shah}
\address{Department of Mathematics, Columbia University, New York, NY 10027, USA}\email{snshah@math.columbia.edu}

\thanks{A.P.\ has been supported by NSF grant DMS-1401858.  S.S.\ has been supported by NSF grants DGE-1148900 and DMS-1401967.}

\begin{abstract}
Inspired by a construction by Bump, Friedberg, and Ginzburg of a two-variable integral representation on $\GSp_4$ for the product of the standard and spin $L$-functions, we give two similar multivariate integral representations.  The first is a three-variable Rankin-Selberg integral for cusp forms on $\PGL_4$ representing the product of the $L$-functions attached to the three fundamental representations of the Langlands $L$-group $\SL_4(\C)$.  The second integral, which is closely related, is a two-variable Rankin-Selberg integral for cusp forms on $\PGU(2,2)$ representing the product of the degree 8 standard $L$-function and the degree 6 exterior square $L$-function.
\end{abstract}

\maketitle

\section{Introduction}

Integral representations of Rankin-Selberg type have been used over the years to relate the analytic behavior of $L$-functions to the study of the more tractable analytic properties of Eisenstein series.  The vast majority of Rankin-Selberg integral representations relate an Eisenstein series in a single variable to a single $L$-function together with some normalizing factors, which usually take the form of Dirichlet $L$-functions.

In the last couple of decades, there has been interest in multivariate versions of this type of construction.   Bump and Friedberg \cite{bf} gave a two-variable integral on $\GL_n$ representing the product of the standard and exterior square $L$-functions.  Bump, Furusawa, and Ginzburg \cite{bfg} later gave a two-variable integral on $\GL_{3n}$ unfolding to a nonunique model that represents the product of the standard $L$-function and the dual of standard $L$-function.  Bump, Friedberg, and Ginzburg \cite{bfg2} gave several constructions of two-variable integral representations on $\GSp_4, \GSp_6,$ and $\GSp_8$ representing the product of the standard and spin $L$-functions.  These were among the first examples of such identities in more than one complex variable.

Ginzburg--Hundley \cite{ginzh} found the first \emph{three}-variable Rankin-Selberg integral; for a generic cusp form on the split orthogonal similitude group $\mathrm{GSO}_8$, it represents the product of the standard $L$-function with two spin $L$-functions.  Gan--Hundley \cite{gh} gave a general construction of an integral for quasi-split groups of type $D_4$ that specializes to a different three-variable integral when the group is split.  Continuing the work initiated in \cite{ginzh}, Hundley \cite{hundley} gave various constructions of two-variable integrals on split orthogonal groups.  Recently, Hundley-Shen \cite{hs} gave a two-variable integral on $\GSp_4\times \GL_2\times \GL_2$ representing the product of two $\GL_2$-twisted spin $L$-functions, one coming from each $\GL_2$-factor.  We found in \cite{pollackShah2} a two-variable Rankin-Selberg integral on $\GSp_4 \times \GL_2$.

Such integrals are valuable for many reasons.  Analysis of the Eisenstein series in these constructions give a tool to study the relationship between the different $L$-functions involved; for instance, Bump, Friedberg, and Ginzburg \cite[Theorem A]{bfg2} rule out simultaneous poles at $s=1$ for the standard and spin $L$-functions on $\GSp_6$.  There seems to be an interesting relationship between multivariate Rankin-Selberg integrals and those unfolding to non-unique models, which we have explored in \cite{pollackShah3} and \cite{pollackShah}.  There is an analogous relationship between the $\GU(2,2)$ construction below and a non-unique model integral given by the first named author in \cite{pollack}.

Inspired by the form of the aforementioned Bump-Friedberg-Ginzburg construction on $\GSp_4$ \cite{bfg2}, we give two multivariate Rankin-Selberg integrals that have a similar shape -- one in three variables on $\GL_4$ and one in two variables on $\GU(2,2)$.  The integrand of the construction in \cite{bfg2} is the product of a cusp form with both a Siegel and Klingen Eisenstein series.  Using the symplectic form given by the anti-diagonal matrix
\[J=\(\begin{array}{cccc} & & & 1 \\ & & 1 & \\ & -1 & & \\ -1 & & & \end{array}\),\]
the Siegel and Klingen parabolics on $\GSp_4$ are those represented by matrices with entries of the form
\begin{equation} \label{eqn:siegelklingen} \(\begin{array}{cccc} * & * & * & * \\ * & * & * & * \\ & & * & * \\ & & * & * \end{array}\)\textrm{ and } \(\begin{array}{cccc} * & * & * & * \\ & * & * & * \\ & * & * & * \\ & & & * \end{array}\),\end{equation}
respectively.  Both are maximal, so the Eisenstein series for these parabolics are functions of a single complex variable.  The parabolics with the same shape on $\GL_4$ are respectively maximal and non-maximal, with respective Eisenstein series in one and two complex variables.  On the other hand, the parabolics with the same shape on $\GU(2,2)$ are both maximal and give Eisenstein series in a single variable.  Both the integral representations given here are of the products of the pair of Eisenstein series on $\GL_4$ or $\GU(2,2)$ with a cusp form in a generic automorphic representation.

We give rough statements of the main theorems of the paper.  For precise definitions of the groups, Eisenstein series, etc., see Sections \ref{sec:3var} or \ref{sec:2var} below.  We first give the integral on $\GL_{4/F}$, $F$ a number field.  Everywhere in this paper, $\mathbf{A}$ denotes the adeles $\mathbf{A}_F$ of $F$.  We write $\zeta_F$ for the Dedekind zeta function of $F$.
\begin{thm}
If $\pi$ is a cuspidal automorphic representation on $\GL_{4/F}$ with trivial central character and $E_P(g,w)$ and $E_Q(g,s_1,s_2)$ are the Eisenstein series with degenerate data associated to the parabolics $P$ and $Q$ given by the respective shapes in (\ref{eqn:siegelklingen}), then we have
\begin{align*} &\int_{\GL_4(F)Z(\A)\backslash \GL_4(\A)}\phi(g)E_P(g,w)E_Q(g,s_1,s_2)\,dg\\ =_S &\frac{L(\pi,\mathrm{Std},2w+s_1-s_2-\frac{1}{2})L(\pi,\wedge^2,2s_1+2s_2-1)L(\pi,\wedge^3,2w-s_1+s_2-\frac{1}{2})}{\zeta_F(4w)\zeta_F(4s_1)\zeta_F(4s_2)\zeta_F(4w-1)\zeta_F(4s_1+4s_2-2)},
\end{align*}
where $=_S$ means that the two sides of the equality are both Eulerian (i.e.\ factor into local components for each place) and, moreover, that the local components are the same away from a finite set of places $S$.
\end{thm}
We have the following result for the group $\GU(2,2)$.  Write $E/F$ for the totally imaginary quadratic extension used to define this group.  (See Section \ref{sec:2var} for details.)
\begin{thm}
If $\pi$ is a generic cuspidal automorphic representation on $\GU(2,2)$ with trivial central character and $E_P(g,w)$ and $E_Q(g,s)$ are the Siegel and Klingen Eisenstein series on $\GU(2,2)$ respectively, then we have
\begin{align*}
  &\int_{\GU(2,2)(F)Z(\A)\backslash \GU(2,2)(\A)}\phi(g)E_P(g,w)E_Q(g,s)\,dg=_S\frac{L(\pi,\mathrm{Std},2w-\frac{1}{2})L(\pi,\wedge^2,3s-1)}{L(\epsilon_{E/F},4w-1)\zeta_F(4w)\zeta_E(3s)\zeta_F(6s-2)},
\end{align*}
where $\epsilon_{E/F}$ is the quadratic character of the field extension.
\end{thm}

{\bf Acknowledgements}:  We thank the anonymous referee for suggestions that have improved the writing of this paper, as well as for a simplification of the proof of Lemma \ref{charTheory} below.  We also thank Wee Teck Gan for catching a referencing error in an earlier version of this article.

\section{Three-variable Rankin-Selberg integral for $\GL_4$}\label{sec:3var}

Let $\set{b_1, b_2, b_3, b_4}$ be the ordered basis of $V_4$, the standard representation of $G=\GL_4$ with the \emph{right} action.  Define the standard maximal parabolic $P$ and non-maximal parabolic $Q$ to be the respective stabilizers of the flags $\langle b_3,b_4 \rangle$ and $\langle b_4 \rangle \subseteq \langle b_2,b_3,b_4 \rangle$.  These definitions agree with the shapes given by (\ref{eqn:siegelklingen}) in the introduction.

We define Eisenstein series for $P$ and $Q$ with degenerate data.  For $P$, pick a standard section $f_P(g,w) \in \ind_{P(\mathbf{A})}^{G(\mathbf{A})}\delta_P^w$ that factorizes as $f_P(g,w) = \prod_vf_{P,v}(g_v,w)$, $g_v \in \GL_4(F_v)$.  For $Q$, first denote by $P_{3,1}$ the maximal parabolic stabilizing the line $\langle b_4\rangle$ via the right action, so that the elements of $P_{3,1}$ have $0$'s for their $(4,1)$, $(4,2)$, and $(4,3)$ entries.  Similarly define $P_{1,3}$ to be the parabolic with $0$'s in the spots $(1,2), (1,3),$ and $(1,4)$, so that $Q = P_{3,1} \cap P_{1,3}$.  Then we pick a standard section $f_Q(g,s_1,s_2)$ in $\ind_{Q(\mathbf{A})}^{G(\mathbf{A})}(\delta_{P_{3,1}}^{s_1}\delta_{P_{1,3}}^{s_2})$ that similarly factorizes as $\prod_v f_{Q,v}(g_v,s_1,s_2)$. The Eisenstein series are then $E_P(g,w) = \sum_{\gamma \in P(F) \backslash \GL_4(F)}{f_P(\gamma g,w)}$ and $E_Q(g,s_1,s_2) = \sum_{\gamma \in Q(F) \backslash \GL_4(F)}{f_Q(\gamma g,s_1,s_2)}$, which are absolutely convergent when the real parts of $w, s_1, s_2$ are sufficiently large.  If $\pi$ is a cuspidal representation, and $\phi$ is in the space of $\pi$, the global integral is
\[I(\phi,s_1,s_2,w) = \int_{\GL_4(F)Z(\A)\backslash \GL_4(\A)}{\phi(g)E_P(g,w)E_Q(g,s_1,s_2)\,dg}.\]
We will show that the integral unfolds to the Whittaker model of $\phi$, and that it represents the product of the $L$-functions of the three fundamental representations of $\SL_4$.

\subsection{Global construction}
Set
\[\nu_P=\(\begin{array}{cc|cc} & 1& & \\  1& & & \\ \hline  & & &1\\  & & 1& \end{array}\) \textrm{ and }\nu_Q=\(\begin{array}{c|cc|c} 1& & & \\ \hline & &1 & \\ & 1& &\\ \hline & & &1 \end{array}\).\]
The following lemma will be used in the unfolding.
\begin{lemma}\label{bruhat} The double coset space $Q(F) \backslash \GL_4(F) \slash P(F)$ is represented by the four permutation matrices corresponding to the permutations $\{\nu_P\nu_Q^{-1}=(1243),1,\gamma_1 = (123),\gamma_2= (243)\}$.  For $i \in \set{1,2}$, the stabilizer of the coset $Q(F)\gamma_i$ inside $P(F)$ contains the unipotent radical of a parabolic subgroup of $\GL_4$.
\end{lemma}
\begin{proof} The Bruhat decomposition gives $Q(F) \backslash \GL_4(F) \slash P(F) = \langle (23) \rangle \backslash S_4 \slash \langle (12), (34) \rangle$.  The coset space $S_4 \slash \langle (12), (34) \rangle$ is represented by the six permutations $\sigma \in S_4$ with $\sigma(1) < \sigma(2)$ and $\sigma(3) < \sigma(4)$.  Letting $\langle (23) \rangle$ act on this set, one finds 4 distinct lists of the form $(\sigma(1),\sigma(2), \sigma(3),\sigma(4))$, namely $(1,2,3,4)$, $(1,4,2,3)$, $(2,3,1,4)$ and $(2,4,1,3)$.  These correspond to the permutations listed.  The stabilizer for $\gamma_1$ contains the unipotent radical of the parabolic $P_{3,1}$ and the stabilizer for $\gamma_2$ contains the unipotent radical of $P_{1,3}$.
\end{proof}

Define $R = (\nu_P^{-1}Q\nu_P) \cap (\nu_Q^{-1}P\nu_Q)\subseteq \GL_4$.  Then $R$ consists of matrices of the form
\begin{equation}\label{eqn:rform}\(\begin{array}{cc|cc} *& &* &* \\  &* & &* \\ \hline  & &* &\\  & & &* \end{array}\).\end{equation}
Define $R^0$ to be the unipotent radical of $R$.
\begin{proposition}\label{prop:unfoldinggl4} Denote by
\[W_\phi^\chi(g) = \int_{U_B(F) \backslash U_B(\mathbf{A})} \chi^{-1}(u)\phi(ug)\]
the Whittaker model of $\phi$ attached to a non-degenerate character $\chi: U_B(F) \backslash U_B(\mathbf{A}) \rightarrow \mathbf{C}^\times$.  Then
\[I(\phi,s_1,s_2,w) = \int_{R^0(\A)Z(\A)\backslash \GL_4(\A)}{W_\phi^\chi(g)f_P(\nu_Qg,w)f_Q(\nu_Pg,s_1,s_2)\,dg}.\]
\end{proposition}
\begin{proof} Unfolding $E_P$ and then $E_Q$ gives, using Lemma \ref{bruhat} and the cuspidality of $\phi$,
\[I(\phi,s_1,s_2,w) = \int_{\((\nu_Q\nu_P^{-1} Q(F) \nu_P\nu_Q^{-1}) \cap P(F)\) Z(\A)\backslash \GL_4(\A)}{\phi(g)f_P(g,w)f_Q(\nu_P\nu_Q^{-1}g,s_1,s_2)\,dg}.\]
Changing variables, this is
\[\int_{R(F) Z(\A)\backslash \GL_4(\A)}{\phi(g)f_P(\nu_Qg,w)f_Q(\nu_Pg,s_1,s_2)\,dg}.\]
Forming an inner integral over $R^0(F)\backslash R^0(\A)$ and Fourier expanding along the rest of the unipotent radical of the Borel gives the result.\end{proof}
For $v$ a place of $F$, define
\[I_v(W,s_1,s_2,w) = \int_{R^0(F_v)Z(F_v)\backslash G(F_v)}{W^{\chi,v}_\phi(g)f_P(\nu_Qg,w)f_Q(\nu_Pg,s_1,s_2)\,dg}.\]
Also write $\zeta_{F_v}$ for the local factor at $v$ of the $\zeta$-function of $F$.

\subsection{Unramified computation}
In this section we prove the following theorem.
\begin{theorem}\label{thm:unramgl4} When all the data is unramified and $v$ is non-Archimedean,
\[I_v(W,s_1,s_2,w) = \frac{L(\pi_v,\mathrm{Std},2w+s_1-s_2-\frac{1}{2})L(\pi_v,\wedge^2,2s_1+2s_2-1)L(\pi_v,\wedge^3,2w-s_1+s_2-\frac{1}{2})}{\zeta_{F_v}(4w)\zeta_{F_v}(4s_1)\zeta_{F_v}(4s_2)\zeta_{F_v}(4w-1)\zeta_{F_v}(4s_1+4s_2-2)}.\]
\end{theorem}
\begin{proof} Let $T$ denote the diagonal torus of $\GL_4$, a typical element of which is $t = \diag(t_1,t_2,t_3,t_4)$, and let $U_B$ denote the unipotent radical of the upper-triangular Borel of $\GL_4$.  Since everything is local, we write $F$ for $F_v$, $I$ for $I_v$, $W$ for $W^{\chi,v}_\phi$, $f_P$ and $f_Q$ for $f_{P,v}$ and $f_{Q,v}$, $p$ for a uniformizer of $\mathcal O_F$, and $|\,\,|$ for the absolute value on $F$ such that $|p|$ is the size of $\mathcal O_F$ modulo its maximal ideal. Then $I(W,s_1,s_2,w)$ is equal to
\begin{align*}& \int_{Z\backslash T}\delta_B^{-1}(t)W(t)\int_{R^0\backslash U_B}{\chi(u)f_P(\nu_Qut,w)f_Q(\nu_Put,s_1,s_2)\,du}\,dt\\ =&\int_{Z\backslash T}\delta_B^{-1}(t)W(t)\left|\frac{t_1}{t_4}\right|\int_{R^0\backslash U_B}{\chi(tut^{-1})f_P(\nu_Qtu,w)f_Q(\nu_Ptu,s_1,s_2)\,du}\,dt \\ =& \int_{Z\backslash T}\delta_B^{-1}(t)W(t)\left|\frac{t_1}{t_4}\right|\left|\frac{t_1t_3}{t_2t_4}\right|^{2w}\left|\frac{t_2t_1t_4}{t_3^3}\right|^{s_1}\left|\frac{t_2^3}{t_1t_4t_3}\right|^{s_2}\int_{R^0\backslash U_B}{\chi(tut^{-1})f_P(\nu_Qu,w)f_Q(\nu_Pu,s_1,s_2)\,du}\,dt. \end{align*}
Parametrize $Z\backslash T$ by $(a,b,c) \in F^\times \times F^\times \times F^\times$ with $(a,b,c)$ mapping to the diagonal element $t(a,b,c) = \diag(abc,bc,c,1)$.  As usual, set $K = \delta_B^{-1/2}W$.  With $t = t(a,b,c)$, then
\[\delta_B^{-1}(t)W(t)\left|\frac{t_1}{t_4}\right|\left|\frac{t_1t_3}{t_2t_4}\right|^{2w}\left|\frac{t_2t_1t_4}{t_3^3}\right|^{s_1}\left|\frac{t_2^3}{t_1t_4t_3}\right|^{s_2} = K(t(a,b,c))|ac|^{2w-\frac{1}{2}}\left|\frac{ab^2}{c}\right|^{s_1-\frac{1}{4}}\left|\frac{b^2c}{a}\right|^{s_2-\frac{1}{4}},\]
and the vanishing properties of $W$ imply this expression is 0 unless $a,b,c \in \mathcal O_F$.

With $t= t(a,b,c)$ and $a,b,c \in \mathcal O_F$, we compute the inner integral.  For $r \in F$, let $n_{ij}(r) = 1 + rE_{ij}$ and let $E_{ij}$ denote the matrix with $1$ at the $(i,j)$ position and $0$'s elsewhere.  Then $R^0 \backslash U_B$ is isomorphic to $F \times F \times F$ via the map sending $(x,y,z)$ to $n_{12}(x)n_{23}(y)n_{34}(z)$.  Observe that the two-by-two matrix $\mm{1}{}{r}{1}$ is right $K$-equivalent to $\mm{1}{r^{-1}}{}{1} \mm{r^{-1}}{}{}{r}$ when $|r| > 1$.  It follows that $f_P(\nu_Q u,w) = f_P(\nu_Q n_{23}(y),w) = |y|^{-4w}$ if $|y| > 1$ and is $1$ if $|y| \leq 1$.  Similarly, $f_Q(\nu_Pu,s_1,s_2) = f_Q(\nu_P n_{12}(x),s_1,s_2)f_Q(\nu_P n_{34}(z),s_1,s_2)$.  The section $f_Q(\nu_P n_{12}(x),s_1,s_2)$ is 1 when $|x| \leq 1$ and is $|x|^{-4s_2}$ when $|x| > 1$, and $f_Q(\nu_P n_{34}(z),s_1,s_2)$ is 1 when $|z| \leq 1$ and $|z|^{-4s_1}$ when $|z| > 1$.

The inner integral is thus
\begin{align*}
  & \left(1 + \int_{|y| > 1}{\psi(by)|y|^{-4w}\,dy}\right)\left(1 + \int_{|x| > 1}{\psi(ax)|x|^{-4s_2}\,dx}\right)\left(1 + \int_{|z| > 1}{\psi(cz)|z|^{-4s_1}\,dz}\right)\\ =&\frac{1}{\zeta(4w)\zeta(4s_1)\zeta(4s_2)}\left(\frac{1-|p|^{(m+1)(4w-1)}}{1-|p|^{4w-1}}\right)\left(\frac{1-|p|^{(\ell+1)(4s_2-1)}}{1-|p|^{4s_2-1}}\right)\left(\frac{1-|p|^{(n+1)(4s_1-1)}}{1-|p|^{4s_1-1}}\right),\end{align*}
where $\ell, m, n$ are the $p$-adic valuations of $a,b,c$.

Let us write $K[\ell,m,n]$ for $K(t(p^\ell,p^m,p^n))$. We deduce that $I(W,s_1,s_2,w)$ is
\begin{align*}\frac{1}{\zeta(4w)\zeta(4s_1)\zeta(4s_2)}&\sum_{\ell,m,n \geq 0}{K[\ell,m,n]|p|^{(\ell+n)(2w-\frac{1}{2})}|p|^{(\ell+2m-n)(s_1-\frac{1}{4})}|p|^{(-\ell+2m+n)(s_2-\frac{1}{4})}} \\ & \times \left(\frac{1-|p|^{(m+1)(4w-1)}}{1-|p|^{4w-1}}\right)\left(\frac{1-|p|^{(\ell+1)(4s_2-1)}}{1-|p|^{4s_2-1}}\right)\left(\frac{1-|p|^{(n+1)(4s_1-1)}}{1-|p|^{4s_1-1}}\right).\end{align*}
Set $X= |p|^{(2w-\frac{1}{2})+(s_1-\frac{1}{4})-(s_2-\frac{1}{4})}$, $Y = |p|^{2[(s_1-\frac{1}{4})+(s_2-\frac{1}{4})]}$, and $Z =|p|^{(2w-\frac{1}{2})-(s_1-\frac{1}{4})+(s_2-\frac{1}{4})}$.  Then $XZ = |p|^{4w-1}$, $\frac{YZ}{X} = |p|^{4s_2-1}$, $\frac{XY}{Z} = |p|^{4s_1-1}$, and the above is the reciprocal of the zeta factors times
\[\sum_{\ell, m,n \geq 0}{K[\ell,m,n]Y^m\left(\frac{1-(XZ)^{m+1}}{1-XZ}\right)\left(\frac{X^{\ell+1}-(YZ)^{\ell+1}}{X-YZ}\right)\left(\frac{Z^{n+1}-(XY)^{n+1}}{Z-XY}\right)}.\]
By Lemma \ref{charTheory} below and the Casselman-Shalika formula, this is
\[\left(\sum_{t,v \geq 0}{A_\pi[t,0,v]X^tZ^v}\right)\left(\sum_{u \geq 0}{A_\pi[0,u,0]Y^u}\right)\]
where $A_\pi[n_1,n_2,n_3]$ is the character $A[n_1,n_2,n_3]$ defined in the statement of Lemma \ref{charTheory} evaluated on the conjugacy class associated to $\pi_v$ in $\SL_4(\C)$.  This product is
\[\frac{L(\pi_v,\mathrm{Std},2w+s_1-s_2-\frac{1}{2})L(\pi_v,\wedge^3,2w-s_1+s_2-\frac{1}{2})}{\zeta_{F_v}(4w-1)}\frac{L(\pi_v,\wedge^2,2s_1+2s_2-1)}{\zeta_{F_v}(4s_1+4s_2-2)},\]
so the theorem follows.
\end{proof}
Denote by $\omega_1, \omega_2, \omega_3$ the fundamental weights of $\SL_4$, so that $\omega_1$ is the highest weight of the standard representation, $\omega_2$ is the highest weight of $\wedge^2$, and $\omega_3$ is the highest weight of $\wedge^3$ or the dual to the standard representation.
\begin{lem}\label{charTheory}Write $A[n_1,n_2,n_3]$ for the character of the irreducible representation of $\SL_4$ with highest weight $n_1\omega_1 + n_2\omega_2 + n_3\omega_3$.  Then one has the identity of power series
\begin{align}\nonumber \sum_{\ell, m,n \geq 0}A[\ell,m,n]Y^m & \left(\frac{1-(XZ)^{m+1}}{1-XZ}\right)\left(\frac{X^{\ell+1}-(YZ)^{\ell+1}}{X-YZ}\right)\left(\frac{Z^{n+1}-(XY)^{n+1}}{Z-XY}\right)\\
\label{eqn:charTheory} =& \left(\sum_{t,v \geq 0}{A[t,0,v]X^tZ^v}\right)\left(\sum_{u \geq 0}{A[0,u,0]Y^u}\right)\end{align}
\end{lem}
\begin{proof} We first claim that the coefficient of $X^tY^uZ^v$ on the right-hand side of (\ref{eqn:charTheory}) is
\begin{equation} \label{eqn:charTheoryclaim1}A[t,0,v]A[0,u,0] = \sum_{\substack{0 \le k \le i \le t\\ 0 \le j \le u-i \\ u - v \le j+k}} A[t+u-2i-j,i+j-k,v-u+j+2k].\end{equation}
This is obtained by applying the Littlewood-Richardson rule.  The terms on the right hand side correspond to strict extensions of the Young diagram $D_{t,0,v}$ with rows of size $(t+v,v,v)$ by the Young diagram $D_{0,u,0}$ with rows of size $(u,u)$.  The parameters $i,j,$ and $k$ correspond respectively to the number of boxes labeled 1 added to the second row of $D_{t,0,v}$, the number of boxes labeled 2 added to the second row of $D_{t,0,v}$, and the number of boxes labeled 2 added to the third row of $D_{t,0,v}$.  These three parameters completely determine the extension since the remainder of the 1's and 2's must go in the top row and fourth row respectively.  The inequalities and the weight of the representation are easily computed from the resulting diagram.

The coefficient of $A[\ell,m,n]$ on the left-hand side of (\ref{eqn:charTheory}) can be rewritten as
\begin{equation} \label{eqn:gseries} \(\sum_{\alpha = 0}^\ell X^{\ell-\alpha}Y^\alpha Z^\alpha\)\(\sum_{\beta=0}^m X^\beta Y^m Z^\beta\) \(\sum_{\gamma=0}^n X^\gamma Y^\gamma Z^{n-\gamma}\) = \sum_{\substack{\alpha\in[0,\ell]\\\beta\in[0,m]\\\gamma\in[0,n]}} X^{\ell-\alpha+\beta+\gamma}Y^{m+\alpha+\gamma}Z^{n+\alpha+\beta-\gamma}.\end{equation}
We need to identify the left-hand side of (\ref{eqn:charTheory}) with
\begin{equation} \label{eqn:goal} \sum_{t,u,v \ge 0} \sum_{\substack{0 \le k \le i \le t\\ 0 \le j \le u-i \\ u - v \le j+k}} A[t+u-2i-j,i+j-k,v-u+j+2k].\end{equation}

Rearranging the summation in (\ref{eqn:charTheory}) using (\ref{eqn:gseries}) and the change of variable $\ell = \alpha+\alpha',m=\beta+\beta',$ and $n=\gamma+\gamma'$, we obtain
\begin{equation} \label{eqn:rearr} \sum_{\alpha,\beta,\gamma,\alpha',\beta',\gamma' \ge 0} A[\alpha+\alpha',\beta+\beta',\gamma+\gamma']X^{\alpha'+\beta+\gamma}Y^{\alpha+\beta+\beta'+\gamma}Z^{\alpha+\beta+\gamma'}.\end{equation}
If one instead starts with (\ref{eqn:goal}), applies the substitutions $\alpha = u-i-j$, $\alpha' = t-i$, $\beta = i-k$, $\beta'=j$, $\gamma = k$, and $\gamma' = v-u+j+k$, and rearranges the resulting system of inequalities, one again arrives at (\ref{eqn:rearr}).
\end{proof}

\section{Two-variable Rankin-Selberg integral for $\GU(2,2)$} \label{sec:2var}

In this section we give the two-variable integral on $\GU(2,2)$ mentioned above.  This integral is the quasi-split version of the three-variable integral on $\GL_4$. Since many of the computations are identical to those in Section \ref{sec:3var}, we only provide brief proofs.

\subsection{Notation} Let $E/F$ be a quadratic extension of fields.  Denote by $J_4$ the matrix
\[J_4=\(\begin{array}{cc|cc} & & &1 \\  & &1 & \\ \hline  &-1 & &\\  -1 & & & \end{array}\).\]
We define $G = \GU(2,2)_{/F}$ to be the algebraic group over $F$ consisting of the elements $(g,\nu(g)) \in (\mathrm{Res}_F^E\GL_4) \times \GL_{1/F}$ that satisfy $g J_4 {}^*g = \nu(g) J_4$, where ${}^*g$ denotes conjugate transpose.

We denote by $(W_4,J_4)$ the four dimensional skew-Hermitian vector space over $E$, which is the defining representation of $G$.  The group $G$ acts on the right of $W_4$.  We write $\set{e_1, e_2, f_2, f_1}$ for the ordered basis of $W_4$ and write $\langle \;,\; \rangle$ for the skew-Hermitian form associated to $J_4$.  Thus $\langle \lambda w_1, \mu w_2\rangle = \lambda \overline{\mu} \langle w_1, w_2 \rangle$ for $w_1, w_2 \in W_4$ and $\lambda, \mu$ in $E$.  Moreover, $\langle e_i, f_j \rangle = \delta_{ij} = -\langle f_j, e_i \rangle$ and $\langle e_i, e_j \rangle = \langle f_i, f_j \rangle = 0$ for all $i,j \in \set{1,2}$.

\subsection{Dual groups and $L$-functions}\label{Lfcns}  The integral representation studied in this section will produce a degree $6$ (exterior square) and degree $8$ (standard) $L$-function on $\mathrm{PGU}(2,2)$. We define the dual groups and $L$-functions, and then provide more explicit descriptions at places where $\pi_v$ is unramified.

\subsubsection{Dual groups}
Recall that the dual group of $G =\GU(2,2)$ is
\[{}^LG=(\GL_1(\C) \times \GL_4(\C)) \rtimes \gal(E/F).\]
The nontrivial element $\theta$ of $\gal(E/F)$ acts on $\GL_1(\C) \times \GL_4(\C)$ by
\[(\lambda, g) \mapsto (\lambda \det(g), \Phi_4 {}^{t}g^{-1}\Phi_4^{-1}),\]
where
\[\Phi_4 = \(\begin{array}{cccc} & & & 1 \\ & & -1 & \\ & 1 & & \\ -1 & & & \end{array}\).\]
(See, for instance, \cite{skinner}.)  The adjoint group $\mathrm{PGU}(2,2)$ of $G$ has dual group given by the derived subgroup $\SL_4(\C) \rtimes \gal(E/F)$ of ${}^LG$.

\subsubsection{Exterior square representation of $^{L}G$}

We now define the ``exterior square'' representation of $^{L}G$.  A very clear discussion of this representation is given in \cite[Section 2]{fm}.  For the convenience of the reader, we now recall this discussion, following \cite{fm} closely.

Let $V_4$ denote the standard $4$-dimensional representation of $\GL_4(\C)$, and let ${}^LG^\circ=\GL_1(\C) \times \GL_4(\C)$ act on $\wedge^2 V_4 \cong \mathbf{C}^6$ via $(\lambda, g) \mapsto \lambda \wedge^2\!(g)$.  Denote this representation by $\rho^\circ: {}^LG^\circ \rightarrow\GL_6(\mathbf{C})$.  Then $x \mapsto \rho^\circ(x)$ and $x \mapsto \rho^\circ(\theta^{-1} x \theta)$ are irreducible representations of $\GL_1(\C) \times \GL_4(\C)$ with the same highest weight, and thus are isomorphic.  Hence, there is an element $A \in \GL(\wedge^2 V_4)$ satisfying $\rho^\circ(\theta^{-1} x \theta) = A^{-1} \rho^\circ(x) A.$  Since $A$ and $-A$ induce the same conjugation action, we can and do choose $A$ to have positive trace.

For example, if one uses the ordered bases
\begin{equation} \label{eqn:vbasis} \set{v_1,v_2,v_3,v_4}\textrm{ for }V_4\textrm{ and }\set{v_1 \wedge v_2, v_1 \wedge v_3, v_1 \wedge v_4, v_2 \wedge v_3,v_2 \wedge v_4, v_3 \wedge v_4}\textrm{ for }\wedge^2 V_4,\end{equation}
then $A$ may be represented by the block diagonal matrix
\begin{equation} \label{eqn:defa} A=\diag\(\mathbf{1}_2,\mb{}{1}{1}{},\mathbf{1}_2\),\end{equation}
where $\mathbf{1}_2$ is the $2\times 2$ identity matrix.

The map that sends $x = (\lambda, g)$ to $\rho^\circ(x)$ and $(1,1) \rtimes \theta$ to $A$ defines a representation $\rho:{}^LG \rightarrow \GL_6(\mathbf{C})$, where ${}^{L}G = {}^LG^\circ \rtimes \gal(E/F) = (\GL_1(\C) \times \GL_4(\C)) \rtimes \gal(E/F)$.  This representation is what we refer to as the exterior square; we denote it by $\wedge^2$.

\begin{remark} Had we chosen $A$ with negative trace, we would get a \emph{different} six-dimensional representation $\wedge^{2}_\mathrm{neg}$ of ${}^{L}G$.  The results below pertain to $\wedge^2$, not $\wedge^{2}_\mathrm{neg}$. \end{remark}

The standard representation of ${}^LG$ has a simpler definition.  If $x = (\lambda,g) \in {}^LG^\circ$, let $\rho_\mathrm{std}^\circ(x) = \lambda g \in \GL(V_4)$.  Then the standard representation $\rho_\mathrm{std}$ is the $8$-dimensional representation $\rho_\mathrm{std}:{}^LG \rightarrow \GL_8(\mathbf{C})$ defined by $\rho_\mathrm{std} = \ind_{{}^{L}G^\circ}^{{}^{L}G}\rho_\mathrm{std}^\circ$. 

\subsubsection{Unramified $L$-factors}

Suppose that for a finite place $v$, $\pi_v$ is an unramified representation of $\mathrm{GU}(2,2)$.  We consider two cases.

{\bf Case 1} ($v$ splits in $E$):
If $v$ splits in $E$, then there is an isomorphism $\mathbf{Q}_v \otimes_\mathbf{Q} E \rightarrow \mathbf{Q}_v \oplus \mathbf{Q}_v$.  This identification induces projections $p_i:\mathrm{GU}(2,2)_{/\mathbf{Q}_v}\rightarrow \GL_{4/\mathbf{Q}_v}$ corresponding to each factor $i \in \set{1,2}$, where for a $\mathbf{Q}$-algebra $R$ and group $G$ over $\mathbf{Q}$ we write $G_{/R}$ for its scalar extension to $R$.  We obtain an isomorphism $\mathrm{GU}(2,2)_{/\mathbf{Q}_v} \cong \GL_{1/\mathbf{Q}_v} \times \GL_{4/\mathbf{Q}_v}$ defined by mapping $g$ to the pair $(\mu(g),p_1(g))$.  This induces a unique identification between the dual groups of $\GU(2,2)_{/\mathbf{Q}_v}$ and $\GL_{1/\mathbf{Q}_v} \times \GL_{4/\mathbf{Q}_v}$, so we may use ${}^{L}G^{\circ} = \GL_1(\C) \times \GL_4(\C)$ to define the local $L$-factor.

Thus for $v$ split, we have
\[L(\pi_v,\wedge^2,s) = L(\pi_v',\lambda \times \wedge^2,s)\textrm{ and }L(\pi_v,\mathrm{Std},w) = L(\pi_v',\lambda \times \mathrm{Std},w)L(\pi_v',\lambda \times \wedge^3,w).\]
Here, $\pi_v'$ denotes the representation $\pi_v$ thought of as a representation of $\GL_1 \times \GL_4$, the notation $\lambda \times \rho$ denotes the tensor product of the tautological 1-dimensional representation of $\GL_1(\mathbf{C})$ with a representation $\rho$ of $\GL_4$, and the $L$-functions of $\pi_v'$ are those considered in Section \ref{sec:3var}.

{\bf Case 2} ($v$ is inert in $E$):
If the finite place $v$ is inert in $E$ and $\pi_v$ is unramified, then the exterior square and standard $L$-functions are related to $L$-functions for an embedded symplectic group.  Let $T$ denote the diagonal torus of $G$.  Define $G' = \GSp_4 \subseteq G$ to be the group of matrices $g \in \GL_{4/F}$ that satisfy $g J_4 {}^tg = \nu(g) J_4$ for some $\nu(g) \in F^\times$, and denote by $T' = T \cap G'$ the diagonal torus of $G'$.

The dual group of $G'$ is ${}^LG'=\mathrm{GSpin}_5(\C)$.  We write Spin for its $4$-dimensional spin representation; under the isomorphism $\mathrm{GSpin}_5(\C) \simeq \GSp_4(\C)$, this is the defining four-dimensional representation of $\GSp_4(\C)$.  We denote by Std the five-dimensional representation of $\mathrm{GSpin}_5(\C)$ that sits inside the exterior square of the Spin representation.  Letting $\mu: \mathrm{GSpin}_5(\C) \rightarrow \GL_1(\C)$ denote the similitude, $\text{Std} \otimes \mu^{-1}$ factors through $\mathrm{GSpin}_5(\C) \rightarrow \SO_5(\C)$ and has an invariant quadratic form.

The following well-known proposition relates $L$-functions for $G$ to $L$-functions for $G'$ when $v$ is inert in $E$.
\begin{prop}\label{gsp4L} Suppose that $v$ is inert in $E$, $\alpha$ is an unramified character of $T$, and $\pi_v$ is an unramified irreducible subquotient of $\ind_B^{G}(\delta_B^{1/2}\alpha)$.  Denote by $\alpha'$ the restriction of $\alpha$ to $T'$, and suppose that $\pi_v'$ is an unramified irreducible subquotient of $\ind_{B'}^{G'}(\delta_{B'}^{1/2}\alpha')$.  Then if $\pi_v$ has central character $\omega_\pi$,
\[L(\pi_v, \wedge^2,s) = L(\pi_v',\mathrm{Spin},s)L(\omega_\pi,2s)\]
and
\[L(\pi_v,\mathrm{Std},w) = \frac{L(\pi_v',\mathrm{Std},2w)}{L(\omega_\pi,2w)}.\] \end{prop}
The first part of this proposition is implicit in \cite{fm}.  Let $p=p_v$ be a uniformizer of $F_v$.  Note that the reciprocals of both sides of the first equality are of degree $6$ polynomials in $|p|_{F_v}^{s}$, and that the reciprocals of both sides of the second equality are degree four polynomials in $|p|_{F_v}^{2w}$ and thus degree eight polynomials in $|p|_{F_v}^{w}$.  We omit the proof of Proposition \ref{gsp4L}, which follows easily from the explicit choices of bases and the matrix $A$ given in (\ref{eqn:vbasis}) and (\ref{eqn:defa}) above.

\subsection{The global integral}
We define the Siegel parabolic $P$ of $G$ to be the stabilizer of the (isotropic) subspace $\langle f_1, f_2 \rangle \subseteq W_4$, and the Klingen parabolic $Q$ to be the stabilizer of the line $\langle f_1 \rangle$.  These are consistent with the shapes (\ref{eqn:siegelklingen}) from the introduction.  If $f_P \in \ind_{P(\A)}^{G(\A)}\delta_P^w$ and $f_Q \in \ind_{Q(\A)}^{G(\A)}\delta_Q^s$ are standard sections with respective factorizations $\prod_v f_{P,v}$ and $\prod_v f_{Q,v}$, then the Siegel and Klingen Eisenstein series are, respectively,
\[E_P(g,w) = \sum_{\gamma \in P(F) \backslash G(F)}{f_P(\gamma g,w)} \qquad \text{ and } \qquad E_Q(g,s) = \sum_{\gamma \in Q(F) \backslash G(F)}{f_Q(\gamma g,s)}.\]
The global integral is
\[I(\phi,s,w) = \int_{G(F)Z_{E}(\A)\backslash G(\A)}{\phi(g)E_P(g,w)E_Q(g,s)\,dg}.\]

The unfolding of $I(\phi,s,w)$ is identical to that of the three-variable integral considered in Section \ref{sec:3var}.  Namely, set
\[\nu_P=\(\begin{array}{cc|cc} & 1& & \\  1& & & \\ \hline  & & &1\\  & & 1& \end{array}\) \textrm{ and } \nu_Q=\(\begin{array}{c|cc|c} 1& & & \\ \hline  & &1 & \\  & -1& &\\ \hline  & & &1 \end{array}\).\]
Define $R = (\nu_P^{-1}Q\nu_P) \cap (\nu_Q^{-1}P\nu_Q)$.  Then $R$ is the subgroup of $G$ consisting of matrices of the form in (\ref{eqn:rform}).  Denote by $U_B \subseteq B$ the unipotent radical of the upper-triangular Borel of $G$, and suppose $\chi: U_B(F) \backslash U_B(\A) \rightarrow \C^\times$ is a nondegenerate character of $U_B$.  Define
\[W_\phi^\chi(g) = \int_{U_B(F) \backslash U_B(\A)}{\chi^{-1}(u) \phi(ug) \,du}\]
to be the $\chi$-Whittaker model of $\phi$.
\begin{prop} Suppose $\pi$ has trivial central character.  Denote by $Z_E$ the center of $G$, which is $E^\times$ embedded diagonally, and let $R^0$ be the unipotent radical of $R$. Then for any choice of nondegenerate character $\chi$ of $U_B$,
\begin{equation} \label{unfoldedInt}I(\phi,s,w) = \int_{R^0(\A)Z_E(\A) \backslash G(\A)}{W_\phi^\chi(g)f_P(\nu_Qg,w)f_Q(\nu_Pg,s) \,dg}.\end{equation}
\end{prop}
\begin{proof} One proves that the double coset $Q(F) \backslash G(F) \slash P(F)$ is represented by $\{1,\nu_P\nu_Q^{-1}\}$, and then the unfolding proceeds exactly as in the proof of Proposition \ref{prop:unfoldinggl4}.
\end{proof}

\subsection{Unramified calculation} As is well-known, the global Whittaker coefficient $W^\chi_\phi$ factorizes over the places $v$ of $F$ as a product of local Whittaker functions whenever $\phi$ is a pure tensor in a tensor-product decomposition of $\pi$.  Thus the integral on the right-hand side of (\ref{unfoldedInt}) is an Euler product.  We now compute the local analogues of this integral when all the data is unramified.

More precisely, assume that $v$ is a finite place of $F$ unramified in $E$, write $E_v = F_v \otimes_F E$, and write $p=p_v$ for a uniformizer for $F_v$.  Also write $\mathcal{O}_{F_v}$ for the ring of integers of $F_v$.  Denote by $n(x_1,x_2)$ an element of $G(F_v)$ of the form
\[\left(\begin{array}{cccc}1 & x_1 & *&* \\ & 1& x_2 & * \\ & & 1 &* \\ & & & 1 \end{array}\right)\]
for $x_1 \in E_v$, $x_2 \in F_v$.  Denote by $\chi_v$ the restriction of $\chi$ to $U_B(F_v)$.  Assume that the character $\chi_v$ is unramified in the sense of \cite[pg.\ 219]{cs}.  Set $K_v = G(\mathcal{O}_{F_v})$.  Assume the representation $\pi_v$ is spherical, and denote by $W_v$ the unique element of $\ind_{U_B(F_v)}^{G(F_v)}\chi_v$ in the space of the Whittaker model of $\pi_v$ with $W_v(K_v) = 1$.  Finally, assume $f_P(K_v,w) = f_Q(K_v,s) = 1$. 

Define
\[I_v(W_v,s,w) = \int_{R(F_v)Z_E(F_v) \backslash G(F_v)}{f_{P,v}(\nu_Qg,w)f_{Q,v}(\nu_Pg,s)W_v(g) \,dg}.\]
\begin{theorem} Denote by $\epsilon_{E/F}:F^\times\backslash \mathbf{A}_F^\times$ the quadratic character associated to $E/F$, i.e.\ $\epsilon_{E/F}(p_v)=1$ if the place $v$ of $F$ splits in $E$ and $\epsilon_{E/F}(p_v)=-1$ if $v$ is inert in $E$.  (Recall that $p_v=p$ is a uniformizer of $F_v$.)  Then when the data is unramified,
\[I_v(W_v,s,w) = \frac{L(\pi_v,\mathrm{Std},2w-\frac{1}{2})L(\pi_v,\wedge^2,3s-1)}{L(\epsilon_{E/F,v},4w-1)\zeta_{F_v}(4w)\zeta_{E_v}(3s)\zeta_{F_v}(6s-2)}.\]
Here, we have the notation
\[L(\epsilon_{E/F,v},w) = (1-\epsilon_{E/F}(p)|p|_v^s)^{-1} \textrm{ and } \zeta_{E_v}(s) = \prod_{w|v} (1 - |p|_w^s)^{-1}.\]
\end{theorem}
\begin{proof}  If $v$ splits in $E$, this becomes the unramified calculation considered in Section \ref{sec:3var}.  Thus we briefly explain the proof in the case that $v$ is inert in $E$. Write $q$ for the order of the residue field of the integer ring of $F_v$.  Set $U = q^{-(2w-\frac{1}{2})}$, $V = q^{-(3s-1)}$, and $K^G_{\pi_v} = \delta_B^{-1/2}W_v$, where $W_v$ is, as above, the spherical Whittaker function of $\pi_v$ normalized so that $W_v(1) = 1$.  Define $I_1(W_v,s,w) = \zeta_{F_v}(4w)\zeta_{E_v}(3s)I(W_v,s,w)$. By manipulations similar to those in the proof of Theorem \ref{thm:unramgl4}, we get that
\[I_1(W_v,s,w) = \sum_{n, m \geq 0}{V^{n}U^{2m}\left(\frac{1-U^{2n+2}}{1-U^2}\right)\left(\frac{1-V^{2m+2}}{1-V^2}\right)K^G_{\pi_v}[m,n]},\]
where $K^G_{\pi_v}[m,n] = K^G_{\pi_v}(\diag(p^{m+n},p^{n},1,p^{-m}))$.  Since the rational root system of $G$ is that of $\GSp_4$, a fact due to Tamir \cite{tamir} is that the Casselman-Shalika formula for $G$ reduces to that of $\GSp_4$ in the sense that $K_{\pi_v}^G = K_{\pi_v'}^{\GSp_4}$ on the torus $T'$ of $\GSp_4 \subseteq G$.  Here $\pi_v'$ and $T'$ are as in Proposition \ref{gsp4L}.  See \cite[Section 5.2]{gh} for an exposition of this relationship.  Applying this fact, $I_1$ is computed in \cite[Theorem 1.2]{bfg2} to be
\begin{equation}\label{bfg comp}\frac{\bm{L}(\pi_v',\mathrm{Std},U^2)\bm{L}(\pi_v',\mathrm{Spin},V)}{\bm{\zeta}_{F_v}(U^4)}.\end{equation}
Here, $\bm{L}(\pi_v,\mathrm{Std},Z)$ means $\sum_{k \geq 0}{\tr(\mathrm{Sym}^k \mathrm{Std}(A_{\pi_v}))Z^k}$, where $A_{\pi_v}$ is the conjugacy class in $^LG$ corresponding to $\pi_v$, and similarly for the other $\bm{L}$ and $\bm{\zeta}$-functions. Via Proposition \ref{gsp4L}, (\ref{bfg comp}) is equal to
\[\frac{\bm{\zeta}(U^2)\bm{L}(\pi_v,\mathrm{Std},U)\bm{L}(\pi_v,\wedge^2,V)}{\bm{\zeta}(U^4)\bm{\zeta}(V^2)} = \frac{L(\pi_v,\mathrm{Std},2w-\frac{1}{2})L(\pi_v,\wedge^2,3s-1)}{L(\epsilon_{E/F,v},4w-1)\zeta(6s-2)}.\]
The theorem follows.
\end{proof}

\bibliography{integralRepnBib}
\bibliographystyle{plain}
\end{document}